\documentclass[reqno]{amsart}
\usepackage[utf8x]{inputenc}  

\usepackage{setspace}
\usepackage[left=3.2cm, right=3.2cm, bottom=4cm]{geometry}                 

\usepackage{graphicx} 
\usepackage{pgf,tikz}
\usetikzlibrary{arrows}
\usepackage{amssymb}
\usepackage{pdfsync}
\usepackage{mathrsfs}
\usepackage{hyperref} 
\usepackage{verbatim} 
\usepackage{epstopdf}
\usepackage{bbm} 
\usepackage[colorinlistoftodos,prependcaption,textsize=tiny]{todonotes}
\usepackage{xargs}

\def\R{\mathbb{R}}

\def\Z{\mathbb{Z}}
\def\C{\mathbb{C}}
\def\H{\mathbb{H}}
\def\SS{\mathbb{S}}

\def\H{\mathcal{H}}

\newcommandx{\emanuel}[2][1=]{\todo[linecolor=green,backgroundcolor=green!25,bordercolor=black,#1]{#2}}

\newcommandx{\diogo}[2][1=]{\todo[linecolor=orange,backgroundcolor=orange!25,bordercolor=orange,#1]{#2}}

\newcommandx{\mateus}[2][1=]{\todo[linecolor=blue,backgroundcolor=blue!25,bordercolor=blue,#1]{#2}}

\newcommandx{\danger}[2][1=]{\todo[linecolor=red,backgroundcolor=red!25,bordercolor=blue,#1]{#2}}

\renewcommand{\d}{\text{\rm d}}

\newtheorem{theorem}{Theorem}

\newtheorem{lemma}[theorem]{Lemma}

\numberwithin{equation}{section}

\allowdisplaybreaks

\title{Sharp mixed norm spherical restriction}

\author[Carneiro]{Emanuel Carneiro}
\author[Oliveira e Silva]{Diogo Oliveira e Silva}
\author[Sousa]{Mateus Sousa}

\address{
IMPA - Instituto de Matem\'{a}tica Pura e Aplicada\\
Rio de Janeiro - RJ, Brazil, 22460-320.}
\email{carneiro@impa.br}
\email{mateuscs@impa.br}
\address{
        Hausdorff Center for Mathematics\\
        53115 Bonn, Germany}
\email{dosilva@math.uni-bonn.de}

\date{\today}                                           
\begin{document}

\subjclass[2010]{42B10}
\keywords{Fourier restriction, sphere, extremizers, optimal constants, delta calculus, Bessel functions, mixed norm.}
\begin{abstract} 
Let $d\geq 2$ be an integer and let $2d/(d-1) < q \leq \infty$. In this paper we investigate the sharp form of the mixed norm Fourier extension inequality
\begin{equation*}
\big\|\widehat{f\sigma}\big\|_{L^q_{{\rm rad}}L^2_{{\rm ang}}(\R^d)} \leq {\bf C}_{d,q}\, \|f\|_{L^2(\SS^{d-1},\d\sigma)},
\end{equation*}
established by L. Vega in 1988. Letting 
$\mathcal{A}_d \subset (2d/(d-1), \infty]$ be the set of exponents for which the constant functions on $\SS^{d-1}$ are the unique extremizers of this inequality, we show that: (i) $\mathcal{A}_d$ contains the even integers and $\infty$; (ii) $\mathcal{A}_d$ is an open set in the extended topology; (iii) $\mathcal{A}_d$ contains a neighborhood of infinity $(q_0(d), \infty]$ with $q_0(d) \leq \left(\tfrac{1}{2} + o(1)\right) d\log d$. In low dimensions we show that $q_0(2) \leq 6.76\,;\,q_0(3) \leq 5.45 \,;\, q_0(4) \leq 5.53 \,;\, q_0(5) \leq 6.07$. In particular, this breaks for the first time the even exponent barrier in sharp Fourier restriction theory. The crux of the matter in our approach is to establish a hierarchy between certain weighted norms of Bessel functions, a nontrivial  question of independent interest within the theory of special functions.
\end{abstract}

\maketitle

\section{Introduction}
This paper is mainly inserted in the fields of harmonic analysis and special function theory. It deals with some sharp mixed norm inequalities within the topic of Fourier restriction. As we shall see, the optimal form of such inequalities is related to certain integrals involving Bessel functions. We start by describing the main results and terminology of the paper and, at the end of this introduction, we provide a brief overview of the related literature.

\subsection{Main results} Let $d\geq 2$ be an integer and let $2d/(d-1) < q \leq \infty$. In his doctoral thesis \cite{VThesis, Vega92} L. Vega proved the following inequality:
\begin{equation}\label{Vega}
\big\|\widehat{f\sigma}\big\|_{L^q_{{\rm rad}}L^2_{{\rm ang}}(\R^d)}:=
\left(\int_0^{\infty}\left(\int_{\SS^{d-1}}\big|\widehat{f\sigma}(r\omega)\big|^2\d\sigma_\omega\right)^{q/2}r^{d-1}\,\d r\right)^{1/q}\leq {\bf C}_{d,q}\, \|f\|_{L^2(\SS^{d-1},\d\sigma)},
\end{equation}
where $\sigma = \sigma_{d-1}$ denotes the canonical surface measure on the unit sphere $\SS^{d-1} \subset \R^d$, and the Fourier transform of the singular measure $f\sigma$ is defined by
\begin{equation}\label{20171001_4:40pm}
\widehat{f\sigma}(x) = \int_{\SS^{d-1}} e^{-ix\cdot \xi}f(\xi)\,\d\sigma_\xi.
\end{equation}
The example $f \equiv 1$ shows that the requirement $ 2d/(d-1) < q$ is necessary for this mixed norm Fourier extension inequality. An alternative proof of \eqref{Vega} was recently given by A. C\'{o}rdoba in \cite{Cor}. 

In this paper we seek to determine the sharp constant in \eqref{Vega} and characterize its (complex-valued) extremizers, i.e. we wish to calculate the value
\begin{equation*}\label{sharpconstdefi}
    {\bf C}_{d,q}:=\sup_{f\neq 0} \frac{\big\|\widehat{f\sigma}\big\|_{L^q_{{\rm rad}}L^2_{{\rm ang}}(\R^d)}}{\|f\|_{L^2(\SS^{d-1})}},
\end{equation*}
and describe the functions $0 \neq f\in L^2{(\SS^{d-1})}$ that verify
$$\big\|\widehat{f\sigma}\big\|_{L^q_{{\rm rad}}L^2_{{\rm ang}}(\R^d)}={\bf C}_{d,q}\,\|f\|_{L^2(\SS^{d-1})}.$$
Throughout the paper, we let $\H^d_{k}$ denote the space of spherical harmonics of degree $k$ in the sphere $\SS^{d-1}$, and let $J_{\nu}$ denote the classical Bessel function of the first kind. In this paper, in particular, we shall always be dealing with $\nu \geq 0$. For $d\geq 2$ and $2d/(d-1) < q < \infty$ we define, for each $k \in \{0,1,2,\ldots\}$, the Bessel integral 
\begin{equation}\label{Intro_Bessel}
\Lambda_{d,q}(k) := \left(\int_0^{\infty} \big|r^{-\frac{d}{2} +1} \, J_{\frac{d}{2} - 1 + k}(r)\big|^q\,r^{d-1}\,\d r\right)^{1/q}.
\end{equation}
For $q = \infty$ we define
\begin{equation*}
\Lambda_{d,\infty}(k) = \sup_{r \geq 0} \big|r^{-\frac{d}{2} +1} \, J_{\frac{d}{2} - 1 + k}(r)\big|.
\end{equation*}

Our first result is the following.

\begin{theorem} \label{Thm1} Let $d\geq 2$ and $2d/(d-1) < q \leq \infty$. The following propositions hold.
\begin{itemize}
\item[(i)] The sequence $\{\Lambda_{d,q}(k)\}_{k \geq 0}$ verifies $\lim_{k \to \infty} \Lambda_{d,q}(k) = 0$ and we have
$$ {\bf C}_{d,q} =  (2\pi)^{d/2}\max_{k\geq0} \Lambda_{d,q}(k).$$
Moreover, $f$ is an extremizer of \eqref{Vega} if and only if $f \in \H^d_{\ell}$, where $\ell$ is such that 
$$\Lambda_{d,q}(\ell) = \max_{k\geq0} \Lambda_{d,q}(k).$$

\item[(ii)] If $q$ is an even integer or $q = \infty$, then
$$ {\bf C}_{d,q} = (2\pi)^{d/2} \Lambda_{d,q}(0),$$
and the constant functions are the unique extremizers of \eqref{Vega}. 

\smallskip

\item[(iii)] For a fixed dimension $d$, the set $\mathcal{A}_d = \{q \in (2d/(d-1), \infty]\,:\, \Lambda_{d,q}(0) > \Lambda_{d,q}(k) \ {\rm for}\ {\rm all} \ k \geq1\}$, for which the constant functions are the unique extremizers of \eqref{Vega}, is an open set in the extended topology\footnote{That is, the topology generated by the open sets  of the form $(a,b)$ and $(a,\infty]$.}. From {\rm (ii)}, $\mathcal{A}_d$ contains the even integers and the point at infinity and, in particular, we may define
\begin{equation}\label{20171012_7:07pm}
q_0(d) := \inf \{r: \ 2d/(d-1) < r < \infty \ {\rm and} \ (r,\infty] \subset \mathcal{A}_d\}.
\end{equation} 
\end{itemize}
\end{theorem}

From Theorem \ref{Thm1} it is clear that the crux of the matter in this problem is to establish which of the Bessel integrals in \eqref{Intro_Bessel} is the largest, a nontrivial  question of independent interest in special function theory. Numerical simulations suggest that $\Lambda_{d,q}(0)$ is a good candidate to be the largest, and here we are able to prove this fact for all even integers $q$, besides the case $q = \infty$. An interesting feature of our proof for even exponents $q$ is that it takes a detour from the Bessel world, rewriting these integrals using spherical harmonics and relying on a decisive application of delta calculus\footnote{See \cite[Appendix A]{FOS17} for a short introduction to delta calculus.} and the theory of orthogonal polynomials. Part (ii) of Theorem \ref{Thm1} extends a result of Foschi and Oliveira e Silva \cite{FOS17}, who proved that the constant functions are the unique extremizers of \eqref{Vega} in the particular case $(d,q) = (2,6)$.

Our second result provides an effective bound that will be useful for our purposes.

\begin{theorem}\label{Thm2} Let $d \geq 2$ and $2d/(d - \frac{4}{3}) < q \leq \infty$. Let 
\begin{equation}\label{Landau}
{\bf L} := \sup_{\nu>0\,;\, r>0} \big|r^{1/3}\, J_\nu(r) \big| = 0.785746\ldots
\end{equation}
Then
\begin{equation}\label{eq:UpperBound}
\Lambda_{d,q}(k) \leq {\bf L}^{\frac{ k+\frac{d}{q}}{\frac{d}{2}+k-\frac{2}{3}} }\, 2^{\frac{\left(\frac{d}{2}+k-1\right)\left(-\frac{d}{2}+\frac{2}{3}+\frac{d}{q}\right)}{\frac{d}{2}+k-\frac{2}{3}}}\Gamma\big(\tfrac{d}{2}+k\big)^{\frac{-\frac{d}{2}+\frac{2}{3}+\frac{d}{q}}{\frac{d}{2}+k-\frac{2}{3}}} \!\left(\frac{1}{kq+d}+\frac{1}{q\big(\frac{d}{2}-\frac{2}{3}\big)-d}\right)^{1/q}\!.
\end{equation} 
The upper bound above, which we call $U_{d,q}(k)$, is a decreasing function of $k \in \{0,1,2,\ldots\}$ for fixed $d$ and $q$.
\end{theorem}
\noindent {\sc Remark:} The numerical value of the constant ${\bf L}$ was found by L. Landau \cite{Landau00}. The symbol ${\bf L}$ will be reserved for this constant throughout the paper.
 
We have two applications in mind for Theorem \ref{Thm2}. In the first one, we provide good quantitative estimates for the neighborhood at infinity of the set $\mathcal{A}_d$ defined in \eqref{20171012_7:07pm}. For this application we use Theorem \ref{Thm2} to estimate the gap between $\Lambda_{d,\infty}(0)$ and the second largest element in the series $\{\Lambda_{d,\infty}(k)\}_{k \geq 0}$ (note that $q =\infty$ is allowed in Theorem \ref{Thm2}). Below we use the standard notation $o(1)$ to represent a quantity that goes to zero as the dimension $d$ goes to infinity. 

\begin{theorem}\label{Thm3} For $d \geq 2$ we have
$$q_0(d) \leq \left(\tfrac{1}{2} + o(1)\right) d\,\log d.$$
In low dimensions we have the following explicit bounds
\begin{align*}
q_0(2)& \leq 6.76 \ \ ; \ \ q_0(3) \leq 5.45\ \ ;\  \ q_0(4) \leq 5.53 \ \ ;\ \ q_0(5)\leq 6.07  \ \ ;\ \ q_0(6)\leq 6.82 \  ;\\
q_0(7) & \leq 7.70 \ \ ;\  \ q_0(8) \leq 8.69  \ \ ;\ \ q_0(9)\leq 9.78 \ \ ;\ \ q_0(10)\leq 10.95.
\end{align*}
\end{theorem}

\noindent{\sc Remark}: Inequality \eqref{20171012_2:30am} in the proof of Theorem \ref{Thm3} allows one to find an upper bound for $q_0(d)$ for any fixed dimension $d$. 

A few features of Theorem \ref{Thm3} are worthy of highlights. Firstly, it is an interesting question whether one can bring down the order of magnitude of the asymptotic upper bound for $q_0(d)$ or, even with the order $d \log d$, whether one can bring down the value of the leading constant $1/2$. Of course, it may just be possible that $q_0(d)$ is close to the lower endpoint $2d/(d-1)$, but our techniques are still far from yielding bounds \footnote{We use Vinogradov's notation $f = O(g)$ to mean that $|f(d)| \leq C \,|g(d)|$ for a certain constant $C>0$.} of order $O(1)$. A second point of interest is that in low dimensions the upper bounds seem quite good, leaving only a relatively small part of the range as uncharted territory.

Theorem \ref{Thm2} is also useful in the following sense: in the range $q > 2d/(d - \frac{4}{3})$ it simply reduces the problem of finding the largest element in the sequence $\{\Lambda_{d,q}(k)\}_{k \geq 0}$ to a numerical verification. We illustrate this point in our second application, establishing the sharp form of the inequality \eqref{Vega} in the Tomas--Stein endpoint $q = 2(d+1)/(d-1)$ in dimensions $d=4$ and $d=5$ (note that the Tomas--Stein endpoints in dimensions $d=2$ and $d=3$ are already covered by Theorem \ref{Thm1}). Recall that the Tomas-Stein endpoint is the smallest $q$ for which the extension inequality \eqref{Intro_Gen_Rest} holds in the case of the sphere $\SS^{d-1}$. 
\begin{theorem}\label{Thm4} In the cases $(d,q) = (4,\frac{10}{3})$ and $(d,q) =  (5,3)$ we have
$$ {\bf C}_{d,q} =  (2\pi)^{d/2} \Lambda_{d,q}(0),$$
and the constant functions are the unique extremizers of \eqref{Vega}.
\end{theorem}

\subsection{A brief historical perspective} Over the recent years there has been a wave of interesting papers devoted to the study of extremizers of Fourier restriction/extension inequalities on different hypersurfaces $\mathcal{M} \subset \R^d$. The sharp form of the Fourier extension inequality
\begin{equation}\label{Intro_Gen_Rest}
\big\|\widehat{f\mu}\big\|_{L^q(\R^d)} \leq C\, \|f\|_{L^2(\mathcal{M}, \d\mu)},
\end{equation}
(where $\mu$ is a canonical measure on the hypersurface $\mathcal{M} $) with a complete characterization of the extremizers, was only achieved in particular situations when {\it the exponent $q$ is an even integer} (but not yet covering all such possibilities). In this setting, one can use Plancherel's theorem and explore the convolution structure of the problem. The most classical cases are when the hypersurface $\mathcal{M}$ is the paraboloid, the cone, the hyperboloid and the sphere. For the paraboloid and cone, it follows from dilation invariance considerations that the exponent $q$ in \eqref{Intro_Gen_Rest} is uniquely determined as a function of the dimension $d$.

In the case of the paraboloid and cone, all the cases with $q$ even, namely $(d,q) = (2,6), (3,4)$ for the paraboloid and $(d,q) = (3,6), (4,4)$ for the cone, were settled in \cite{Fo07} (see also \cite{BBCH09, BR13, Ca09, Go17, HZ06} for alternative proofs and related problems). The existence of extremizers for \eqref{Intro_Gen_Rest} in all dimensions was established in \cite{Sh09a} for the paraboloid, and in \cite{Ra12} for the cone. 

For the hyperboloid, the works \cite{COSS, Qu15} settle the endpoint cases where $q$ is even, i.e. $(d,q) = (2,6), (3,4), (3,6), (4,4)$, finding the sharp constant $C$ in \eqref{Intro_Gen_Rest} and showing that there are no extremizers in these cases. The sharp form of \eqref{Intro_Gen_Rest} in the non-endpoint cases with $q$ even, i.e. $(d,q) = (2,2m)$, with $m \geq 4$, is still unknown. The existence of extremizers in all non-endpoint cases in dimensions $d=2$ and $d=3$ was established in \cite{COSS}.

Finally, when $\mathcal{M} = \SS^{d-1}$, the existence of extremizers for \eqref{Intro_Gen_Rest} was established in the endpoint cases $(d,q) = (3,4),(2,6)$ in \cite{CS12a} and \cite{Sh16}, respectively, and in all non-endpoint cases in \cite{FVV11}. The sharp form of \eqref{Intro_Gen_Rest} was established only in the Tomas--Stein endpoint case $(d,q) = (3,4)$, in \cite{Fo15}, and in the non-endpoint cases $(d,q) = (4,4), (5,4), (6,4), (7,4)$, in \cite{COS15}. In these situations, the constant functions on the sphere are the unique extremizers. The sharp form of \eqref{Intro_Gen_Rest} in the Tomas--Stein case $(d,q) = (2,6)$ (see the discussion in \cite{CFOST15}), and all the other non-endpoint cases with $q$ even, is currently unknown. 

Breaking the barrier of even exponents in the study of the sharp form of \eqref{Intro_Gen_Rest} is considerably harder and seems to require genuinely new ideas. A historically similar situation surrounded the epic breakthrough of Beckner \cite{Be75} for the Hausdorff--Young inequality. In this context, our results bring interesting new perspectives to this theory. In the restricted Tomas--Stein range $q \geq 2(d+1)/(d-1)$, we note that Vega's inequality \eqref{Vega} is qualitatively weaker than the Fourier extension \eqref{Intro_Gen_Rest} in the case $\mathcal{M} = \SS^{d-1}$. In fact, a simple application of H\"{o}lder's inequality yields
\begin{equation}\label{Intro_3}
\big\|\widehat{f\sigma}\big\|_{L^q_{{\rm rad}}L^2_{{\rm ang}}(\R^d)} \ \leq \ \omega_{d-1}^{1/2-1/q} \, \big\|\widehat{f\sigma}\big\|_{L^q(\R^d)} \ \leq \ C\, \|f\|_{L^2(\SS^{d-1},\d\sigma)},
\end{equation}
where $\omega_{d-1} := \sigma(\SS^{d-1})= 2 \pi^{d/2}\Gamma(d/2)^{-1}$ denotes the total surface area of $\SS^{d-1}$. Therefore, if constants are the unique extremizers of the inequality on the right-hand side of \eqref{Intro_3}, then it follows that they are also the unique extremizers of \eqref{Vega} (in particular, part (ii) of Theorem \ref{Thm1} in the cases $(d,q) = (3,4), (4,4), (5,4), (6,4), (7,4)$ follows from \cite{Fo15} and \cite{COS15}). The opposite implication is not necessarily true, but the fact that the constants extremize \eqref{Vega} in the cases where $q$ is even is in alignment with the possibility that they may be extremizers for the spherical Fourier extension in these cases as well.

\section{Extremizers for Vega's inequality}

In this section we prove Theorem \ref{Thm1}.

\subsection{Revisiting the proof of Vega's inequality} We start by providing a brief direct proof of \eqref{Vega}, from which  the conclusion of part (i) of Theorem \ref{Thm1} will follow. The first ingredient we need is the following precise asymptotic description for weighted $L^q$-norms of Bessel functions, obtained by Stempak in \cite[Eq.\@ (6)]{Stempak00}.

\begin{lemma}[cf. \cite{Stempak00}] \label{Lem4} Let $1 \leq p \leq \infty$ and $\alpha < \frac12 - \frac1p$. Then, as $\nu \to \infty$ we have
\begin{equation}\label{20171001_4:10pm}
\left( \int_0^{\infty} \big| J_{\nu}(r)\,r^{\alpha}\big|^p\,\d r\right)^{1/p} \sim 
\left\{ 
\begin{array}{ll}
\nu^{\alpha - \frac12 + \frac1p}, & 1 \leq p < 4;\\
\nu^{\alpha - \frac14}\,(\log \nu)^{\frac14}, & p=4;\\
\nu^{\alpha - \frac13 + \frac{1}{3p}}, & 4 < p \leq \infty.
\end{array}
\right.
\end{equation}
For $p=\infty$ one takes $\sup_{r>0}\big| J_{\nu}(r)\,r^{\alpha}\big|$ as the $L^{\infty}$-norm. Here $f(\nu) \sim g (\nu)$ as $\nu \to \infty$ stands for $f(\nu) = O(g(\nu))$ and $g(\nu) = O(f(\nu))$ as $\nu\to\infty$. The implicit constants in \eqref{20171001_4:10pm} may depend on the parameters $p$ and $\alpha$.
\end{lemma}

We specialize \eqref{20171001_4:10pm} to our situation by setting $p = q$, $\alpha = \frac{d-1}{q} - \frac{d}{2} + 1$ and $\nu = \frac{d}{2} - 1 + k$. Note that, when $4 < p \leq \infty$, we have $\alpha - \frac13 + \frac{1}{3p} <0$ in our case. It plainly follows from \eqref{20171001_4:10pm} that 
\begin{equation}\label{20171003_2:06pm}
\lim_{k \to \infty} \Lambda_{d,q}(k) = 0.
\end{equation}

The second ingredient we need is a formula for the Fourier transform of a spherical harmonic. Let $Y_k \in \H^d_{k}$ (we shall usually not display the dependence of the spherical harmonic on the dimension $d$, believing this is clear from the context). With our normalization \eqref{20171001_4:40pm} for the Fourier transform we have
\begin{equation}\label{20171003_1:48pm}
\widehat{Y_k\sigma}(x)=(2\pi)^{\frac{d}{2}} \, i^{-k}\, |x|^{-\frac{d}{2}+1}\, J_{\frac{d}{2}-1+k}(|x|)\,Y_k\!\left(\frac{x}{|x|}\right).
\end{equation}
For a proof of this formula we refer the reader to \cite[Chapter IV, Eq. (2.19) and Theorem 3.10]{SteinWeiss}.

With these two ingredients in hand we may prove \eqref{Vega}. Let $0 \neq f\in L^2{(\SS^{d-1})}$ be normalized so that $\|f\|_{L^2(\SS^{d-1})} = 1$. We may expand it as $f = \sum_{k\geq 0} a_k Y_k$, where $a_k \in \C$, $Y_k \in \H^d_{k}$ and $\|Y_k\|_{L^2(\SS^{d-1})} = 1$. From the orthogonality of the subspaces $\{\H^d_{k}\}_{k \geq 0}$ we find that $\sum_{k \geq 0}|a_k|^2 = 1$. In the following computation let us assume that $q \neq \infty$. We use \eqref{20171003_1:48pm} and the orthogonality relations on the second line, Jensen's inequality on the third line, and \eqref{20171003_2:06pm} on the fifth line to get 
\begin{align}\label{20171003_2:12pm}
\big\|\widehat{f\sigma}\big\|^q_{L^q_{{\rm rad}}L^2_{{\rm ang}}(\R^d)}& =
\int_0^{\infty}\left( \int_{\SS^{d-1}}\big|\widehat{f\sigma}(r\omega)\big|^2\d\sigma_\omega\right)^{q/2}r^{d-1}\,\d r \nonumber\\
& = (2\pi)^{qd/2} \int_0^{\infty} \left( \sum_{k \geq 0}  \, r^{-d +2}\, \big|J_{\frac{d}{2}-1+k}(r)\big|^2\,|a_k|^2 \right)^{q/2} r^{d-1}\,\d r\nonumber\\
& \leq (2\pi)^{qd/2} \int_0^{\infty} \left( \sum_{k \geq 0}  \, \big|r^{-\frac{d}{2} +1} \, J_{\frac{d}{2} - 1 + k}(r)\big|^q \,|a_k|^2 \right) r^{d-1}\,\d r\\
& = (2\pi)^{qd/2} \sum_{k \geq 0} \Lambda_{d,q}^q(k) \, |a_k|^2\nonumber\\
& \leq (2\pi)^{qd/2} \left(\max_{k\geq0} \Lambda_{d,q}(k)\right)^q.\nonumber
\end{align}
This establishes inequality \eqref{Vega} with constant 
$$ {\bf C}_{d,q} =  (2\pi)^{d/2}\max_{k\geq0} \Lambda_{d,q}(k).$$
To see that this is indeed the sharp constant, let $\ell \geq0$ be such that $\Lambda_{d,q}(\ell) = \max_{k\geq0} \Lambda_{d,q}(k)$. It is easy to see that $f = Y_\ell$ verifies the equality in the two inequalities of \eqref{20171003_2:12pm}. Finally, let us characterize all the extremizers in the case $q \neq \infty$. The case of equality in the last inequality of \eqref{20171003_2:12pm} is achieved if and only if
$$f = \sum_{j=1}^J a_{\ell_j} Y_{\ell_j},$$
where $\{\ell_j\}_{j=1}^J$ are such that $\Lambda_{d,q}(\ell_j) = \max_{k\geq0} \Lambda_{d,q}(k)$ for all $j =1,2,\ldots, J$ (note that this is a finite set of indices by \eqref{20171003_2:06pm}). In the event that $J \geq 2$ and two or more of the $\{a_{\ell_j}\}_{j=1}^J$ are nonzero, for almost every $r \in (0,\infty)$ (in fact we just need a set of positive measure) we have strict inequality in the third step of \eqref{20171003_2:12pm} (this follows from the fact that the set of points $r>0$ where $\big|J_{\frac{d}{2} - 1 + k}(r)\big| = \big|J_{\frac{d}{2} - 1 + \ell}(r)\big|$, for $k\neq \ell$, consists of isolated points). This implies that the extremizers need to be a single spherical harmonic, which concludes the proof of part (i) of Theorem \ref{Thm1} (modulo the case $q = \infty$ that we treat next).

\subsection{The case $q = \infty$} If $f \in L^1(\SS^{d-1})$ then $\widehat{f\sigma}$ is a continuous function in $\R^d$ which goes to zero at infinity. Using the triangle inequality and H\"{o}lder's inequality we have 
\begin{align}\label{20171003_4:04pm}
\big\|\widehat{f\sigma}\big\|_{L^{\infty}_{{\rm rad}}L^2_{{\rm ang}}(\R^d)} & = \sup_{r \geq 0} \left(\int_{\SS^{d-1}}\big|\widehat{f\sigma}(r\omega)\big|^2\d\sigma_\omega\right)^{1/2} \nonumber \\
& \leq \omega_{d-1}^{1/2}\, \big\|\widehat{f\sigma}\big\|_{L^{\infty}(\R^d)} \\
& \leq \omega_{d-1}^{1/2}\,\|f\|_{L^1(\SS^{d-1})} \nonumber \\
& \leq \omega_{d-1}\,\|f\|_{L^2(\SS^{d-1})}. \nonumber
\end{align}
This establishes inequality \eqref{Vega} with constant 
$$ {\bf C}_{d,\infty} =   \omega_{d-1} = \frac{2 \,\pi^{d/2}}{\Gamma\big(\tfrac{d}{2}\big)} = (2\pi)^{d/2} |r^{-\frac{d}{2} +1} \, J_{\frac{d}{2} - 1}(r)\big| \Big|_{r=0} = (2\pi)^{d/2} \Lambda_{d,\infty}(0).$$
The latter two identities are classical in Bessel function theory, see Lemma \ref{BesselFacts1} below. The example $f \equiv 1$ shows that \eqref{20171003_4:04pm} is indeed sharp.

Now let us characterize all the extremizers. The equality in the the third inequality of \eqref{20171003_4:04pm} (H\"{o}lder's inequality) occurs if and only if $|f|$ is constant a.e. in $\SS^{d-1}$. Suppose that $\big\|\widehat{f\sigma}\big\|_{L^{\infty}(\R^d)}$ is attained at a point $x_0 \in \R^d$ (we know it is attained somewhere since $\widehat{f\sigma}$ is continuous and goes to zero at infinity). In order to have equality in the second inequality of \eqref{20171003_4:04pm}, 
$$\big\|\widehat{f\sigma}\big\|_{L^{\infty}(\R^d)} = \big| \widehat{f\sigma}(x_0) \big| = \left| \int_{\SS^{d-1}} e^{-ix_0\cdot \xi} \,f(\xi)\,\d\sigma_\xi\right| \leq  \int_{\SS^{d-1}} |f(\xi)|\,\d\sigma_\xi\,,$$
we need to have $f(\xi) = c\,e^{ix_0\cdot \xi}\,g(\xi)$ where $c \in \C$ and $g$ is a nonnegative function. From the fact that $|f|$ is constant, we find that $g$ is constant (say $g \equiv 1$) and hence $f(\xi) = c\,e^{ix_0\cdot \xi}$. In this case $\widehat{f\sigma}(x) = c \,\widehat{\sigma}(x - x_0)$. Using the explicit form of $\widehat{\sigma}$ given by \eqref{20171003_1:48pm}, and observing that it attains its maximum modulus only at the origin, one arrives at the conclusion that equality in the first inequality on \eqref{20171003_4:04pm} can only occur if $x_0 = 0$, which forces $f$ to be a constant.

In particular, taking $f = Y_k$ with $\|Y_k\|_{L^2(\SS^{d-1})} =1$ and using \eqref{20171003_1:48pm}, we get
\begin{align*}
\big\|\widehat{f\sigma}\big\|_{L^{\infty}_{{\rm rad}}L^2_{{\rm ang}}(\R^d)} & = \sup_{r \geq 0} \left(\int_{\SS^{d-1}}\big|\widehat{f\sigma}(r\omega)\big|^2\d\sigma_\omega\right)^{1/2} = (2\pi)^{d/2} \Lambda_{d,\infty}(k) < (2\pi)^{d/2} \Lambda_{d,\infty}(0)
\end{align*}
for all $k \geq 1$, since the constants are the only extremizers in this case.  

\subsection{The case $q$ even} We now move to the proof of part (ii) of Theorem \ref{Thm1}, in the case of even exponents $q$. Throughout this proof let us write $q = 2m$, with $m \in \Z$ and $m \geq 2$ (in the case $d=2$ we have $m \geq 3$). From part (i) it suffices to show that 
$$\Lambda_{d,2m}(k) < \Lambda_{d,2m}(0)$$
for all $k \geq 1$. An interesting feature of this proof is that we accomplish this without diving into the classical Bessel machinery, as we shall see.

\subsubsection{Delta calculus tools} Let us start by fixing an orthonormal basis $\{Y_{k,\ell}\}_{\ell =1}^{h_{d,k}}$ of $\mathcal{H}_k^d$. Here $h_{d,k}= \binom{d+k-1}{k} - \binom{d+k-3}{k-2}$ denotes the dimension of this vector space (with the convention $h_{d,0}=1$ and $h_{d,1}=d$). If $Y_k \in \mathcal{H}_k^d$ is a generic spherical harmonic of degree $k$ (recall that the dependence of the spherical harmonics on the dimension $d$ is implicit here), from \eqref{20171003_1:48pm} we have, for $r >0$ and $\omega \in \SS^{d-1}$, 
\begin{equation}\label{20171004_1:08pm}
\big|\widehat{Y_k\sigma}(r\omega)\big|^2 = (2\pi)^d \,\big|r^{-\frac{d}{2}+1}\, J_{\frac{d}{2}-1+k}(r)\big|^2\, |Y_k(\omega)|^2.
\end{equation}
Therefore, for any $m$-uple $(j_1, j_2, \ldots, j_m)$ with $1 \leq j_i \leq h_{d,k}$ for $i = 1,2,\ldots, m$, we use \eqref{20171004_1:08pm} and the definition of the Fourier transform to get  
\begin{align}\label{20171004_4:07pm}
& (2\pi)^{md}\Lambda_{d,2m}^{2m}(k) \prod_{i=1}^m  \big|Y_{k,j_i}(\omega)\big|^2  = \int_0^{\infty}\prod_{i=1}^m\big|\widehat{Y_{k,j_i}\sigma}(r\omega)\big|^2\,r^{d-1}\,\d r\nonumber\\
& \ \ = \int_0^{\infty}\prod_{i=1}^m \left(\int_{\SS^{d-1}}\int_{\SS^{d-1}} e^{-ir\omega\cdot (\xi_i - \eta_i)}\,Y_{k,j_i}(\xi_i) \overline{Y_{k,j_i}(\eta_i)} \,\d\sigma_{\xi_i}\, \d\sigma_{\eta_i}\right) r^{d-1}\,\d r\\
& \ \ = \int_0^{\infty} \int_{(\SS^{d-1})^{2m}}e^{-ir\omega\cdot (\sum_{i=1}^m \xi_i - \sum_{i=1}^m \eta_i)}\left(\prod_{i=1}^mY_{k,j_i}(\xi_i) \overline{Y_{k,j_i}(\eta_i)}\right) \, \d\Sigma\,r^{d-1}\,\d r, \nonumber
\end{align}
where $\d\Sigma = \d\sigma_{\xi_1}\ldots\d\sigma_{\xi_m}\d\sigma_{\eta_1}\ldots\sigma_{\eta_m}.$

From \cite[Chapter IV, Lemma 2.8 and Corollary 2.9]{SteinWeiss} we know that the function
\begin{equation}\label{20171004_4:19pm}
Z_k(\xi,\eta) := \sum_{\ell =1}^{h_{d,k}}Y_{k,\ell}(\xi)\, \overline{Y_{k,\ell}(\eta)}
\end{equation}
is real-valued, does not depend on the choice of the orthonormal basis $\{Y_{k,\ell}\}_{\ell =1}^{h_{d,k}}$ and verifies, for any $\xi, \eta \in \SS^{d-1}$,
\begin{equation}\label{20171004_4:20pm}
|Z_k(\xi,\eta)| \leq Z_k(\xi,\xi) = \sum_{\ell =1}^{h_{d,k}} \big| Y_{k,\ell}(\xi)\big|^2 = h_{d,k} \,\omega_{d-1}^{-1}.
\end{equation} 
The function $Z_k(\xi,\eta)$ is simply the reproducing kernel of the finite dimensional space $\mathcal{H}_k^d$. We now sum \eqref{20171004_4:07pm} over all possible $m$-uples $(j_1, j_2, \ldots, j_m)$, and use \eqref{20171004_4:19pm} and \eqref{20171004_4:20pm} to get
\begin{align}\label{20171004_4:18pm}
\begin{split}
(2\pi)^{md}\,\Lambda_{d,2m}^{2m}(k)\, & h_{d,k}^m \,\omega_{d-1}^{-m} = (2\pi)^{md}\,\Lambda_{d,2m}^{2m}(k) \left(\sum_{\ell =1}^{h_{d,k}} \big| Y_{k,\ell}(\omega)\big|^2 \right)^m\\
& = \int_0^{\infty} \int_{(\SS^{d-1})^{2m}}e^{-ir\omega\cdot (\sum_{i=1}^m \xi_i - \sum_{i=1}^m \eta_i)}\left(\prod_{i=1}^mZ_k(\xi_i, \eta_i)\right)  \d\Sigma\,r^{d-1}\d r.
\end{split}
\end{align}

At this point we would like to have an inequality for the integral on the right-hand side of \eqref{20171004_4:18pm} by replacing $Z_k(\xi_i, \eta_i)$ by the larger quantity $Z_k(\xi_i, \xi_i)$, while still keeping the oscillatory exponential factor intact. Although this is generally a flawed idea, in this particular case it turns out that we actually can do that, and this is exactly where the delta calculus tools come into play. Start by observing that the left-hand side of \eqref{20171004_4:18pm} is independent of the point $\omega \in \SS^{d-1}$. We may then integrate \eqref{20171004_4:18pm} over $\SS^{d-1}$, and use Fubini's theorem and \eqref{20171004_4:20pm} to get
\begin{align}\label{20171004_4:56pm}
(2\pi)^{md}\Lambda_{d,2m}^{2m}(k) \, h_{d,k}^m \,\omega_{d-1}^{-m+1} & = \int_{(\SS^{d-1})^{2m}}\left(\prod_{i=1}^mZ_k(\xi_i, \eta_i)\right) \,\boldsymbol{\delta} \left(\sum_{i=1}^m \xi_i - \sum_{i=1}^m \eta_i\right) \d\Sigma \nonumber\\
& \leq \int_{(\SS^{d-1})^{2m}}\left(\prod_{i=1}^m|Z_k(\xi_i, \eta_i)|\right) \,\boldsymbol{\delta}  \left(\sum_{i=1}^m \xi_i - \sum_{i=1}^m \eta_i\right) \d\Sigma \nonumber\\
& \leq \int_{(\SS^{d-1})^{2m}}\left(\prod_{i=1}^m Z_k(\xi_i, \xi_i)\right) \,\boldsymbol{\delta}  \left(\sum_{i=1}^m \xi_i - \sum_{i=1}^m \eta_i\right) \d\Sigma\\
& = h_{d,k}^m \int_{(\SS^{d-1})^{2m}}\left(\prod_{i=1}^m Z_0(\xi_i, \xi_i)\right) \,\boldsymbol{\delta}  \left(\sum_{i=1}^m \xi_i - \sum_{i=1}^m \eta_i\right) \d\Sigma \nonumber\\
& = (2\pi)^{md}\Lambda_{d,2m}^{2m}(0)  \,h_{d,k}^m \,\omega_{d-1}^{-m+1},\nonumber
\end{align}
where $\boldsymbol{\delta}$ denotes the $d$-dimensional Dirac delta. Above we have also used the fact that $Z_0(\xi,\xi) = \omega_{d-1}^{-1}$, and in the last line of \eqref{20171004_4:56pm} we have simply used back the identity of the first line in the case $k=0$. This plainly establishes that 
\begin{equation}\label{20171009_1:38pm}
\Lambda_{d,2m}(k) \leq \Lambda_{d,2m}(0).
\end{equation}

\subsubsection{The strict inequality: orthogonal polynomials tools} To argue that inequality \eqref{20171009_1:38pm} is indeed strict for $k \geq1$ we proceed as follows. The function $Z_k(\xi,\eta)$ has a particularly simple expression in terms of the Gegenbauer (or ultraspherical) polynomials $C_k^{\lambda}$. For $\lambda >0$, these are orthogonal polynomials in the interval $[-1,1]$ with respect to the measure $(1-t^2)^{\lambda -\frac12}\,\d t$ (in particular, $C_k^{\lambda}$ has degree $k$), and are defined by the generating function
\begin{equation}\label{20171005_12:04pm}
(1 -2rt + r^2)^{-\lambda} = \sum_{k=0}^{\infty} C_k^{\lambda}(t)\,r^k.
\end{equation}
From \cite[Chapter IV, Theorem 2.14]{SteinWeiss}, if $d>2$ we have 
\begin{equation}\label{20171005_1:31pm}
Z_k(\xi,\eta) = c_{d,k} \, C_k^{(d-2)/2}(\xi \cdot \eta),
\end{equation}
where $c_{d,k}$ is a constant. From \eqref{20171004_4:20pm} it follows that 
\begin{equation}\label{20171009_3:50pm}
c_{d,k} \,C_k^{(d-2)/2}(1) = h_{d,k}\,\omega_{d-1}^{-1}.
\end{equation}
The case $d=2$ is also important to us, and we can deal with this case directly by choosing the orthonormal basis $Y_{k,1}(\cos \theta, \sin \theta) = \frac{1}{\sqrt{\pi}} \cos k \theta$ and $Y_{k,2}(\cos \theta, \sin \theta) = \frac{1}{\sqrt{\pi}} \sin k \theta$, for $k \geq 1$. We then find 
\begin{equation}\label{20171005_1:32pm}
Z_k(\xi,\eta) = \frac{1}{\pi} \,T_k(\xi\cdot \eta),
\end{equation}
where $\{T_k\}_{k\geq0}$ is the classical family of Chebyshev polynomials that verify $\cos(k \theta) = T_k(\cos \theta)$. These are orthogonal polynomials in the interval $[-1,1]$ with respect to the measure $(1-t^2)^{-\frac12}\,\d t$, and are defined by the generating function
\begin{equation}\label{20171005_12:05pm}
\frac{1-rt}{1 -2rt + r^2} = \sum_{k=0}^{\infty} T_k(t)\,r^k.
\end{equation}
Note that by differentiating \eqref{20171005_12:04pm}, adding $1$ on both sides, and sending $\lambda \to 0^+$ we obtain, from \eqref{20171005_12:05pm}, the relation $T_k(t) = \frac{k}{2}\lim_{\lambda \to 0^+} \frac{C_k^{\lambda}(t)}{\lambda}$, for $k \geq 1$.

It is known that the Gegenbauer polynomial $C_k^{\lambda}$, for $\lambda >0$ and $k \geq 1$, attains its maximum modulus in the interval $[-1,1]$ only at the endpoints $\pm1$, see \cite[Theorems 7.32.1 and 7.33.1]{Sz}. The Chebyshev polynomial $T_k$, for $k \geq 1$, attains its maximum modulus in the interval $[-1,1]$ at the $k +1$ points $\{\cos(\pi j / k)\}_{j=0}^k$ since $T_k(\cos(\pi j / k)) = \cos(\pi j) = \pm1$. From \eqref{20171005_1:31pm} and \eqref{20171005_1:32pm} we also record the fact that
\begin{equation}\label{20171013_4:49pm}
|Z_k(\xi,\eta)|  = |Z_k(\pm\xi,\pm\eta)|
\end{equation}
for any $\xi, \eta \in \SS^{d-1}$. In order to explore these facts, we revisit \eqref{20171004_4:56pm} and introduce a useful intermediate step by applying H\"{o}lder's inequality. In what follows we let $\sigma^{*(n)}$ denote the $n$-fold convolution of the surface measure $\sigma$ with itself. For $n\geq2$, we recall that the measure $\sigma^{*(n)}$ is absolutely continuous with respect to the Lebesgue measure, and we identify it with its Radon-Nykodim derivative, which is a positive radial function supported in the open ball $B_n$ of radius $n$ (see for instance \cite[Lemma 5]{COS15}). In the case $d>2$, proceeding as in \eqref{20171004_4:56pm} with the convention that $\sigma^{*(n)}(x) = \sigma^{*(n)}(|x|)$, we have (note the use of \eqref{20171013_4:49pm} in the third line below)
\begin{align}\label{20171009_1:52pm}
(2\pi)^{md}\Lambda_{d,2m}^{2m}(k) \,&  h_{d,k}^m \,\omega_{d-1}^{-m+1}   = \int_{(\SS^{d-1})^{2m}}\left(\prod_{i=1}^mZ_k(\xi_i, \eta_i)\right) \,\boldsymbol{\delta} \left(\sum_{i=1}^m \xi_i - \sum_{i=1}^m \eta_i\right) \d\Sigma \nonumber\\
&\leq \int_{(\SS^{d-1})^{2m}}\left(\prod_{i=1}^m|Z_k(\xi_i, \eta_i)|\right) \,\boldsymbol{\delta}  \left(\sum_{i=1}^m \xi_i - \sum_{i=1}^m \eta_i\right) \d\Sigma \nonumber\\
& =  \int_{(\SS^{d-1})^{2m}}\left(\prod_{i=1}^m|Z_k(\xi_i, \eta_i)|\right) \,\boldsymbol{\delta}  \left(\sum_{i=1}^m \xi_i + \sum_{i=1}^m \eta_i\right) \d\Sigma \nonumber \\
& \leq  \prod_{i=1}^m \left(\int_{(\SS^{d-1})^{2m}} |Z_k(\xi_i, \eta_i)|^m \,\boldsymbol{\delta}  \left(\sum_{i=1}^m \xi_i + \sum_{i=1}^m \eta_i\right) \d\Sigma\right)^{1/m} \nonumber \\
&  = \int_{(\SS^{d-1})^{2}} |Z_k(\xi, \eta)|^m \, \sigma^{*(2m-2)}(\xi + \eta) \,\d\sigma_{\xi}\,\d\sigma_{\eta}\\
& = c_{d,k}^m \int_{(\SS^{d-1})^{2}} \big|C_k^{(d-2)/2}(\xi \cdot \eta)\big|^m \, \sigma^{*(2m-2)}\big(\sqrt{2 + 2\,\xi\cdot \eta}\big) \,\d\sigma_{\xi}\,\d\sigma_{\eta}\nonumber \\
& = c_{d,k}^m \,\omega_{d-1}\,\omega_{d-2} \int_{-1}^1 \big|C_k^{(d-2)/2}(t)\big|^m \, \sigma^{*(2m-2)}\big(\sqrt{2 + 2\,t}\big)\,(1-t^2)^{(d-3)/2}\,\d t\nonumber\\
& \leq c_{d,k}^m \,\omega_{d-1}\,\omega_{d-2} \int_{-1}^1 C_k^{(d-2)/2}(1)^m \, \sigma^{*(2m-2)}\big(\sqrt{2 + 2\,t}\big)\,(1-t^2)^{(d-3)/2}\,\d t\nonumber\\
& = (2\pi)^{md}\Lambda_{d,2m}^{2m}(0)  \,h_{d,k}^m \,\omega_{d-1}^{-m+1}.\nonumber
\end{align}
In the last identity we have used \eqref{20171009_3:50pm} and reversed the steps above. Now it is clear that the last inequality in \eqref{20171009_1:52pm} is strict for $k \geq 1$. The argument in dimension $d=2$ is analogous, using the orthogonal polynomials $\{T_k\}_{k\geq0}$. This concludes the proof of part (ii) of Theorem \ref{Thm1}.

\subsection{Open set property: a qualitative approach} We now move to the proof of part (iii) of Theorem \ref{Thm1}. Recall that we have defined the set 
$$\mathcal{A}_d = \{q \in (2d/(d-1), \infty]\,:\, \Lambda_{d,q}(0) > \Lambda_{d,q}(k) \ {\rm for}\ {\rm all} \ k \geq1\},$$
consisting of exponents for which the constant functions are the unique extremizers of \eqref{Vega}. We want to show that this is an open set in the extended topology. We present here a brief qualitative bootstrapping argument.

Let us fix $s \in \mathcal{A}_d$, with $2d/(d-1) < s < \infty$. Fix $p$ and $r$ such that $2d/(d-1) < p < s < r < \infty$. For any $q$ with $p < q < r$, a basic interpolation yields
$$\big\|\widehat{f\sigma}\big\|_{L^q_{{\rm rad}}L^2_{{\rm ang}}(\R^d)} \leq \big\|\widehat{f\sigma}\big\|_{L^p_{{\rm rad}}L^2_{{\rm ang}}(\R^d)}^{\theta} \, \big\|\widehat{f\sigma}\big\|_{L^r_{{\rm rad}}L^2_{{\rm ang}}(\R^d)}^{1-\theta}$$
for some $0 < \theta = \theta(q) < 1$. Applying this to $f = Y_k \in \mathcal{H}_k^d$, with $\|Y_k\|_{L^2(\SS^{d-1})} = 1$, yields
\begin{equation}\label{20171005_4:25pm}
\Lambda_{d,q}(k) \leq \Lambda_{d,p}^{\theta}(k)\,\Lambda_{d,r}^{1-\theta}(k).
\end{equation}
From \eqref{20171003_2:06pm} we may choose $k_0$ sufficiently large so that 
$$\Lambda_{d,p}(k) \leq \frac{\Lambda_{d,s}(0)}{4} \ \ {\rm and} \ \ \Lambda_{d,r}(k) \leq \frac{\Lambda_{d,s}(0)}{4}$$ 
for $k > k_0$. Since the function $q \mapsto \Lambda_{d,q}(0)$ is continuous, for any $q$ in a small neighborhood $U$ around $s$ we have
\begin{equation}\label{20171013_5:50pm}
\Lambda_{d,p}(k) \leq \frac{\Lambda_{d,q}(0)}{2} \ \ {\rm and} \ \ \Lambda_{d,r}(k) \leq \frac{\Lambda_{d,q}(0)}{2}
\end{equation}
for $k > k_0$. From \eqref{20171005_4:25pm} and \eqref{20171013_5:50pm} we obtain
$$\Lambda_{d,q}(k)  \leq \frac{\Lambda_{d,q}(0)}{2}$$
for $k > k_0$ and $q \in U$. Hence, when testing for the condition $\Lambda_{d,q}(0) > \Lambda_{d,q}(k) \ {\rm for}\ {\rm all} \ k \geq1$, we may restrict ourselves to the range $k \leq k_0$ for any $q \in U$. Since these $k_0$ conditions are verified when $q=s$, we can certainly find a neighborhood $V$ with $s \in V \subset U$, where they will still be verified.

The case $s = \infty$ follows similar ideas. Let us fix here $p=6$. As in \eqref{20171005_4:25pm}, for any $6 < q < \infty$ we have, for $k \geq 1$, 
\begin{equation}\label{20171005_5:02pm}
\Lambda_{d,q}(k) \leq \Lambda_{d,6}^{\theta}(k)\,\Lambda_{d,\infty}^{1-\theta}(k) < \rho \, \Lambda_{d,6}^{\theta}(0)\,\Lambda_{d,\infty}^{1-\theta}(0)
\end{equation}
for some $0 < \theta = \theta(q) < 1$ and for some fixed $0 < \rho < 1$. Note the use of part (ii) of Theorem \ref{Thm1} above, and the insertion of the factor $\rho <1$, which is possible since we know from  \eqref{20171003_2:06pm} that there exists a gap between $\Lambda_{d,6}(0)$ (resp. $\Lambda_{d,\infty}(0)$) and the second largest $\Lambda_{d,6}(k)$ (resp. $\Lambda_{d,\infty}(k)$). Using the fact that $\Lambda_{d,q}(0) \to \Lambda_{d,\infty}(0)$ as $q \to \infty$, and that $\theta(q) \to 0$ as $q \to \infty$, from \eqref{20171005_5:02pm} we conclude that there exists $q_1$ such that 
$$\Lambda_{d,q}(k) < \left(\rho + \tfrac{1-\rho}{2}\right)  \Lambda_{d,q}(0)< \Lambda_{d,q}(0)$$
for any $q > q_1$ and any $k \geq1$ (the term $\tfrac{1-\rho}{2}$ above bears no particular significance, it could be any $0< \delta$ with $\rho + \delta <1$). This yields the desired neighborhood of infinity and concludes the proof of Theorem \ref{Thm1}. 

\section {Effective bounds for Bessel integrals} In this section we prove Theorem \ref{Thm2}. 

\subsection{Preliminaries}
Our analysis ultimately relies on some classical facts from Bessel function theory. It is worth mentioning that the bounds in this section can potentially be sharpened by using more refined estimates for Bessel functions (for instance using Lemma \ref{Bessel3} below). In what follows we choose a sample of estimates that already convey the main ideas and suffice for our upcoming applications. They also have the advantage of keeping the technicalities within a reasonable limit.

\begin{lemma} [Bessel facts I] \label{BesselFacts1} For all $\nu \geq 0$ and $r >0$ we have 
\begin{equation}\label{20171009_5:01pm}
\big|J_{\nu}(r)\big| \leq \frac{r^{\nu}}{2^{\nu} \,\Gamma(\nu + 1)}.
\end{equation}
Moreover, 
\begin{equation}\label{20171009_9:53pm}
\sup_{r > 0}\big|r^{-\nu} \, J_{\nu}(r)\big| = \lim_{r\to 0^+} \big|r^{-\nu} \, J_{\nu}(r)\big|  = \frac{1}{2^{\nu} \,\Gamma(\nu + 1)}.
\end{equation}
\end{lemma}
\begin{proof}
We refer the reader to Watson's classical treatise \cite[p.~48--49]{Wa66}.
\end{proof}

\subsection{The upper bound for $\Lambda_{d,q}(k)$} Let $ q < \infty$. We start by breaking the integral defining $\Lambda_{d,q}^q(k)$ into two parts
\begin{equation}\label{20171009_5:19pm}
\Lambda_{d,q}^q(k) := \int_0^{\infty} \big|r^{-\frac{d}{2} +1} \, J_{\frac{d}{2} - 1 + k}(r)\big|^q\,r^{d-1}\,\d r = \int_0^a + \int_a^{\infty} =: {\rm I}(a) + {\rm II}(a),
\end{equation}
with a free parameter $a$ to be properly chosen later. We now estimate ${\rm I}(a)$ and ${\rm II}(a)$ separately.

\subsubsection{Estimating ${\rm I}(a)$} Using \eqref{20171009_5:01pm} we have
\begin{align}\label{20171009_5:18pm}
{\rm I}(a)\leq \frac{1}{2^{q(\frac{d}{2}+k-1)}\,\Gamma\big(\frac{d}{2}+k\big)^q}\int_0^a r^{kq+d-1}\,\d r= \frac{a^{kq+d}}{(kq+d)\,2^{q(\frac{d}{2}+k-1)}\,\Gamma\big(\frac{d}{2}+k\big)^q}.   
\end{align}

\subsubsection{Estimating ${\rm II}(a)$} Using \eqref{Landau} we have
\begin{align}\label{20171009_5:13pm}
{\rm II(a)} \leq  {\bf L}^q \int_a^\infty r^{-q(\frac{d}{2}-\frac{2}{3})+d-1}\,\d r= {\bf L}^q \,\frac{a^{-q(\frac{d}{2}-\frac{2}{3})+d}}{q(\frac{d}{2}-\frac{2}{3})-d}.
\end{align}
In order to have integrability in \eqref{20171009_5:13pm} we constrain ourselves to the range
\begin{equation}\label{20171009_6:38pm}
q > \frac{2d}{d - \frac43}.
\end{equation}

\subsubsection{Optimizing the parameter $a$} From \eqref{20171009_5:19pm}, \eqref{20171009_5:18pm} and \eqref{20171009_5:13pm} we find, for any $a>0$, that 
\begin{equation*}
\Lambda_{d,q}^q(k) \leq \frac{a^{kq+d}}{(kq+d)\,2^{q(\frac{d}{2}+k-1)}\,\Gamma\big(\frac{d}{2}+k\big)^q} + {\bf L}^q \,\frac{a^{-q(\frac{d}{2}-\frac{2}{3})+d}}{q\big(\frac{d}{2}-\frac{2}{3}\big)-d}.
\end{equation*}
By calculus, the optimal choice of the parameter $a$ verifies
\begin{align*}
a^{\frac{d}{2}+k-\frac{2}{3}} ={\bf L}\, 2^{\frac{d}{2}+k-1}\, \Gamma\big(\tfrac{d}{2}+k\big),
\end{align*}
which leads us to
\begin{equation}\label{20171009_6:32pm}
\Lambda_{d,q}^q(k) \leq {\bf L}^{\frac{ qk+d}{\frac{d}{2}+k-\frac{2}{3}} }\, 2^{\frac{\left(\frac{d}{2}+k-1\right)\left(-q\big(\frac{d}{2}-\frac{2}{3}\big)+d\right)}{\frac{d}{2}+k-\frac{2}{3}}}\Gamma\big(\tfrac{d}{2}+k\big)^{\frac{-q\big(\frac{d}{2}-\frac{2}{3}\big)+d}{\frac{d}{2}+k-\frac{2}{3}}} \left(\frac{1}{kq+d}+\frac{1}{q\big(\frac{d}{2}-\frac{2}{3}\big)-d}\right).
\end{equation}
It is easy to see that the first, second and fourth factors on the right-hand side of \eqref{20171009_6:32pm} are decreasing functions of the parameter $k \in \{0,1,2, \ldots\}$ (recall that we work in the range \eqref{20171009_6:38pm}). Let us show that the same is true for the third factor as well. For this we need to show that
$$  \Gamma(x)^{x + \frac13} <  \Gamma(x +1)^{x - \frac23}$$ 
for $x \in \tfrac12\Z$ with $x \geq 1$. Since $\Gamma(x +1) = x \Gamma(x)$, this is equivalent to the claim
$$ \Gamma(x) < x^{x - \frac23}.$$
If $x$ is an integer, then
$$\Gamma(x)=(x-1)!\leq (x-1)^{x-1}<x^{x-1}<x^{x-\frac23}.$$
If $x$ is a half-integer, then
$$\Gamma(x)=(x-1)(x-2)\ldots(3/2)(1/2)\Gamma(1/2)\leq (x-1)^{x-\frac32}(1/2)\sqrt{\pi}< x^{x-\frac23}.$$
This shows that the right-hand side of \eqref{20171009_6:32pm} is a decreasing function of $k \in \{0,1,2, \ldots\}$.

Taking $q$-roots on both sides of \eqref{20171009_6:32pm} yields
\begin{equation}\label{20171009_9:15pm}
\Lambda_{d,q}(k) \leq {\bf L}^{\frac{ k+\frac{d}{q}}{\frac{d}{2}+k-\frac{2}{3}} }\, 2^{\frac{\left(\frac{d}{2}+k-1\right)\left(-\frac{d}{2}+\frac{2}{3}+\frac{d}{q}\right)}{\frac{d}{2}+k-\frac{2}{3}}}\Gamma\big(\tfrac{d}{2}+k\big)^{\frac{-\frac{d}{2}+\frac{2}{3}+\frac{d}{q}}{\frac{d}{2}+k-\frac{2}{3}}} \left(\frac{1}{kq+d}+\frac{1}{q\big(\frac{d}{2}-\frac{2}{3}\big)-d}\right)^{1/q}.
\end{equation} 
Note that we may send $q \to \infty$ in \eqref{20171009_9:15pm} to obtain
\begin{equation}\label{20171009_10:04pm}
\Lambda_{d,\infty}(k) \leq {\bf L}^{\frac{ k}{\frac{d}{2}+k-\frac{2}{3}} }\, 2^{\frac{\left(\frac{d}{2}+k-1\right)\left(-\frac{d}{2}+\frac{2}{3}\right)}{\frac{d}{2}+k-\frac{2}{3}}}\Gamma\big(\tfrac{d}{2}+k\big)^{\frac{-\frac{d}{2}+\frac{2}{3}}{\frac{d}{2}+k-\frac{2}{3}}}.
\end{equation}
This completes the proof of Theorem \ref{Thm2}.

\smallskip

\noindent {\sc Remark:} Using Lemma \ref{Bessel3} in the set $r > \frac32 (\frac{d}{2} -1 +k)$ and repeating the procedure above in the set $r \leq \frac32 (\frac{d}{2} -1 +k)$ yields an upper bound for $\Lambda_{d,q}^q(k)$ in the whole range $\frac{2d}{d-1} < q$. However, in the range $\frac{2d}{d-1} < q < \frac{2d}{d - \frac43}$, such upper bound tends to $\infty$ as $k \to \infty$, and in the point $ q =  \frac{2d}{d - \frac43}$ such upper bound tends to a stationary value when $k \to \infty$. Finding explicit bounds in the range $\frac{2d}{d-1} < q \leq \frac{2d}{d - \frac43}$, that are decreasing and tend to zero as $k \to \infty$, seems to be a slightly more complicated task from the technical point of view, and is related to the problem of  finding explicit constants for the corresponding asymptotic estimates in Lemma \ref{Lem4} (due to Stempak \cite{Stempak00}). The tricky point is that Lemma \ref{Lem4} relies on estimates from the classical work of Barcel\'{o} and C\'{o}rdoba \cite{BC}, some of which derive from stationary phase methods. We do not pursue such matters here since, for the applications we have in mind (the Tomas--Stein endpoints and the neighborhood at infinity), the restricted range given by \eqref{20171009_6:38pm} is sufficient. 

\section{The neighborhood at infinity} \label{Sec4_Infinity}

In this section we prove Theorem \ref{Thm3}. 

\subsection{Preliminaries} We start again by collecting some facts from the theory of special functions that shall be relevant in this section. The idea to use the quadratic approximation \eqref{Quadratic_approx} in our argument below was suggested to us by Dimitar Dimitrov.\footnote{An earlier version of this manuscript used a linear approximation in place of \eqref{Quadratic_approx}, leading to the same asymptotic bound for $q_0(d)$, but yielding inferior bounds in low dimensions.}

\begin{lemma}[Bessel facts II and Stirling's formula] \hfill \label{Sec_4_Lem_7}
\begin{itemize}
\item[(i)] For $\nu> 0$ the following explicit formula holds
\begin{align}\label{20171010_1:21am_1}
\int_0^{\infty}\big|J_{\nu}(r)\big|^4\,r^{1-2\nu}\,\d r=\frac{\Gamma(\nu)\,\Gamma(2\nu)}{2\pi\,\Gamma\big(\nu +\tfrac12\big)^2\,\Gamma(3\nu)}.
\end{align}
\item[(ii)] For $\nu \geq0$ we have
\begin{equation}\label{Quadratic_approx}
\Gamma(\nu +1) \left(\frac{r}{2}\right)^{-\nu}\!J_{\nu}(r) \geq 1 - \frac{r^2}{4(\nu +1)}
\end{equation}
in the range $0 \leq r \leq 2(\nu+1)^{1/2}$.
\smallskip
\item[(iii)] $($Stirling's formula$)$ For $x > 0$ we have 
\begin{equation}\label{Stirling}
\Gamma(x) = \sqrt{2\pi}\,x^{x-\frac12}\,e^{-x}\,e^{\mu(x)}\,,
\end{equation}
where the function $\mu(x)$ verifies
$\frac{1}{12x +1} < \mu(x) < \frac{1}{12x}.$
\end{itemize}
\end{lemma}
\begin{proof}
Part (i) can be found in \cite[\S 6.579, Eq. 3]{GR}\footnote{We note that there is a typo in identity \cite[\S 6.579, Eq. 3]{GR}. There is a parameter $a$ on the left-hand side that is not appearing on the right-hand side, and it should appear as $a^{2\nu-2}$. This is in alignment with the following identity \cite[\S 6.579, Eq. 4]{GR}. In our case we have $a=1$.} and part (iii) is proved in \cite{Ro55}. For part (ii) observe that $r \mapsto \Gamma(\nu +1) \left(\frac{r}{2}\right)^{-\nu}\!J_{\nu}(r)$ is an entire function of order $1$ with series expansion
\begin{equation*}
\Gamma(\nu +1)\left(\frac{r}{2}\right)^{-\nu} J_{\nu}(r) =  \sum_{n=0}^{\infty} \frac{(-1)^n \big(\tfrac12 r\big)^{2n}}{n!(\nu +1)(\nu +2)...(\nu+n)}.
\end{equation*}
Note that the approximation proposed in \eqref{Quadratic_approx} is just the quadratic truncation of this Taylor series. To verify the inequality in \eqref{Quadratic_approx} it suffices to establish that consecutive terms $n=2k$ and $n=2k+1$ (for $k \geq 1$) in this Taylor expansion add up to a nonnegative quantity (in the range $0 \leq r \leq 2(\nu+1)^{1/2}$). That is equivalent to showing that
$$\frac{\big(\tfrac12 r\big)^{4k}}{(2k)!(\nu +1)(\nu +2)...(\nu+2k)} \geq \frac{\big(\tfrac12 r\big)^{4k+2}}{(2k+1)!(\nu +1)(\nu +2)...(\nu+2k+1)},$$
which simplifies to
$$1 \geq \frac{\big(\tfrac12 r\big)^{2}}{(2k+1)(\nu+2k+1)}.$$
The latter is obviously true in the range $0 \leq r \leq 2(\nu+1)^{1/2}$.
\end{proof}

\subsection{Estimating $q_0(d)$ for $d \geq 3$}  Recall that $U_{d,q}(k)$ and $U_{d,\infty}(k)$ denote the upper bounds on the right-hand sides of \eqref{20171009_9:15pm} and \eqref{20171009_10:04pm}, respectively. In light of \eqref{20171009_9:53pm}, we start by observing that 
\begin{equation}\label{20171009_10:14pm}
\Lambda_{d,\infty}(0) = U_{d,\infty}(0) = \frac{1}{2^{\frac{d}{2}-1} \,\Gamma\big(\tfrac{d}{2}\big)}.
\end{equation}
\subsubsection{The gap} We already know from Theorem \ref{Thm1} that there exists a gap between $\Lambda_{d,\infty}(0)$ and the rest of the $\Lambda_{d,\infty}(k)$ for $k \geq 1$. The first main ingredient we need is an explicit estimate for such gap. Since $k \mapsto U_{d,\infty}(k)$ is a decreasing function of $k$, we may use \eqref{20171009_10:14pm} to obtain, for $k \geq 1$, the following estimate
\begin{align}\label{20171009_11:11pm}
\frac{\Lambda_{d,\infty}(k)}{\Lambda_{d,\infty}(0)} \leq \frac{U_{d,\infty}(k)}{U_{d,\infty}(0)} \leq \frac{U_{d,\infty}(1)}{U_{d,\infty}(0)} =  \left({\bf L}^{6}\left( \frac{2^{3d-6}\,\Gamma\big(\tfrac{d}{2}\big)^{6}}{d^{3d-4}}\right)\right)^{\frac{1}{3d+2}}=:\beta(d).
\end{align}
For the rest of this subsection we work in dimension $d \geq 3$ and consider $q \geq 4$. From H\"{o}lder's inequality, part (ii) of Theorem \ref{Thm1} and inequality \eqref{20171009_11:11pm} we have, for $k \geq 1$,
\begin{align}\label{20171012_3:20am}
\Lambda_{d,q}(k) \leq \Lambda_{d,4}(k)^{4/q} \,\Lambda_{d,\infty}(k)^{1-4/q} <  \Lambda_{d,4}(0)^{4/q}\,\beta(d)^{1-4/q}\, \Lambda_{d,\infty}(0)^{1-4/q}.
\end{align}
Therefore, in order to show that $q \in \mathcal{A}_d$ for some $q\geq4$, it suffices to verify that 
\begin{equation}\label{20171010_12:32am}
\Lambda_{d,4}(0)^{4/q}\,\beta(d)^{1-4/q}\, \Lambda_{d,\infty}(0)^{1-4/q} < \Lambda_{d,q}(0).
\end{equation}
Note that this is the case for large enough $q$ since $\beta(d) <1$ and $\Lambda_{d,q}(0)\rightarrow \Lambda_{d,\infty}(0)$ as $q\rightarrow \infty$.

\subsubsection{Replacing $\Lambda_{d,q}(0)$ by an explicit lower bound} We shall use \eqref{Quadratic_approx} with $\nu = \frac{d}{2} -1$. In the following computation let us write $\alpha = 2(\nu +1)^{1/2} = 2(d/2)^{1/2}$. From \eqref{Quadratic_approx} we obtain
\begin{align}\label{20171010_2:20pm}
\Lambda_{d,q}(0) & > \left(\int_0^{\alpha} \big|r^{-\frac{d}{2} +1} \, J_{\frac{d}{2}-1}(r)\big|^q\,r^{d-1}\,\d r\right)^{1/q}\nonumber\\
& \geq \frac{1}{2^{\frac{d}{2}-1}\,\Gamma\big(\frac{d}{2}\big)} \left(\int_0^{\alpha}\left(1-\frac{r^2}{\alpha^2}\right)^q\,r^{d-1}\,\d r\right)^{1/q}\nonumber\\
& = \frac{\big(\frac{\alpha^d}{2}\big)^{1/q}}{2^{\frac{d}{2}-1}\,\Gamma\big(\frac{d}{2}\big)}\left(\int_0^{1}\left(1-t\right)^q\,t^{\frac{d}{2}-1}\,\d t\right)^{1/q}\\
& = \frac{\Big(2^{(d-1)}\big(\frac{d}{2}\big)^{d/2}\Big)^{1/q}}{2^{\frac{d}{2}-1}\,\Gamma\big(\frac{d}{2}\big)} \left(\frac{\Gamma(q+1)\Gamma\big(\frac{d}{2}\big))}{\Gamma\big(q+\frac{d}{2}+1\big)}\right)^{1/q}\,,\nonumber
\end{align}
where we have used the relation
$$\int_0^1 (1-t)^{x-1}\,t^{y-1} \,\d t =: B(x,y) = \frac{\Gamma(x)\Gamma(y)}{\Gamma(x+y)}.$$

\subsubsection{Conclusion} From \eqref{20171010_12:32am} and \eqref{20171010_2:20pm} it suffices to estimate $q$ such that 
\begin{equation*}
\Lambda_{d,4}(0)^{4/q}\,\beta(d)^{1-4/q}\, \Lambda_{d,\infty}(0)^{1-4/q} \leq \frac{\Big(2^{(d-1)}\big(\frac{d}{2}\big)^{d/2}\Big)^{1/q}}{2^{\frac{d}{2}-1}\,\Gamma\big(\frac{d}{2}\big)} \left(\frac{\Gamma(q+1)\Gamma\big(\frac{d}{2}\big))}{\Gamma\big(q+\frac{d}{2}+1\big)}\right)^{1/q}.
\end{equation*}
From \eqref{20171009_10:14pm} this simplifies to
\begin{equation}\label{20171010_2:28pm}
\left(\frac{\Lambda_{d,4}(0)}{\beta(d)\, \Lambda_{d,\infty}(0)}\right)^{4/q} \leq \frac{\Big(2^{(d-1)}\big(\frac{d}{2}\big)^{d/2}\Big)^{1/q}}{\beta(d)} \left(\frac{\Gamma(q+1)\Gamma\big(\frac{d}{2}\big))}{\Gamma\big(q+\frac{d}{2}+1\big)}\right)^{1/q}.
\end{equation}
Raising both sides of \eqref{20171010_2:28pm} to the $q$-th power and rearranging terms, this is equivalent to 
\begin{equation}\label{20171012_2:16am}
\left(\frac{\Lambda_{d,4}(0)}{\beta(d)\, \Lambda_{d,\infty}(0)}\right)^4\,\frac{1}{2^{(d-1)}\big(\frac{d}{2}\big)^{d/2}\,\Gamma\big(\frac{d}{2}\big)} \leq \frac{\Gamma(q+1)}{\beta(d)^q \, \Gamma\big(q+\frac{d}{2}+1\big)}.
\end{equation}
Observe now from \eqref{20171010_1:21am_1} that 
\begin{align}\label{20171010_1:21am}
\Lambda_{d,4}^4(0)=\int_0^{\infty}\big|J_{\frac{d}{2}-1}(r)\big|^4\,r^{-d+3}\,\d r=\frac{\Gamma\big(\frac{d}{2}-1\big)\,\Gamma(d-2)}{2\pi \, \Gamma\big(\frac{d}{2}-\frac12\big)^2\,\Gamma\big(\frac{3d}{2}-3\big)}.
\end{align}
Using \eqref{20171009_10:14pm} and \eqref{20171010_1:21am} we rewrite \eqref{20171012_2:16am} as
\begin{equation}\label{20171012_2:30am}
\frac{\Gamma\big(\frac{d}{2}-1\big)\,\Gamma(d-2)}{2\pi \, \Gamma\big(\frac{d}{2}-\frac12\big)^2\,\Gamma\big(\frac{3d}{2}-3\big)}\,\left(2^{\frac{d}{2}-1} \,\Gamma\big(\tfrac{d}{2}\big)\right)^4\, \frac{1}{2^{(d-1)}\big(\frac{d}{2}\big)^{d/2}\,\Gamma\big(\frac{d}{2}\big)} \leq \frac{\Gamma(q+1)}{\beta(d)^{q-4} \, \Gamma\big(q+\frac{d}{2}+1\big)}.
\end{equation}
At this point it should be clear that, for each fixed dimension $d$, there is a threshold exponent $q$ for which \eqref{20171012_2:30am} starts to hold, since the exponential factor on the right-hand side eventually dominates. By taking the logarithm on both sides, and carefully applying Stirling's formula \eqref{Stirling} in the form
$$\log \Gamma(x) = x\log x - x - \tfrac12\log x + O(1),$$
one realizes that the main term in \eqref{20171012_2:30am} is of order $d \log d$ and we arrive at
$$q_0(d) \leq \left(\tfrac{1}{2} + o(1)\right) d\,\log d.$$
Note that $\beta(d) \to e^{-1}$ as $d \to \infty$.

Relation \eqref{20171012_2:30am} can also be used directly to estimate the threshold exponent $q$ for any fixed dimension $d$. A routine numerical evaluation yields the values
\begin{align*}
q_0(3) & \leq 5.45\ \ ;\  \ q_0(4) \leq 5.53 \ \ ;\ \ q_0(5)\leq 6.07  \ \ ;\ \ q_0(6)\leq 6.82 \ ;\\
q_0(7) & \leq 7.70 \ \ ;\  \ q_0(8) \leq 8.69  \ \ ;\ \ q_0(9)\leq 9.78 \ \ ;\ \ q_0(10)\leq 10.95.
\end{align*}

\subsection{The case $d=2$} This case needs to be treated slightly differently since the exponent $4$ is not in the admissible range. We start as in \eqref{20171012_3:20am}, but now considering $q \geq 6$,
\begin{align*}
\Lambda_{2,q}(k) \leq \Lambda_{2,6}(k)^{6/q} \,\Lambda_{2,\infty}(k)^{1-6/q} <  \Lambda_{2,6}(0)^{6/q}\,\beta(2)^{1-6/q}\, \Lambda_{2,\infty}(0)^{1-6/q}.
\end{align*}
Again, in order to show that $q \in \mathcal{A}_2$ for some $q\geq6 $, it suffices to verify that 
\begin{equation*}
\Lambda_{2,6}(0)^{6/q}\,\beta(2)^{1-6/q}\, \Lambda_{2,\infty}(0)^{1-6/q} < \Lambda_{2,q}(0).
\end{equation*}
The rest of the analysis is exactly the same. Instead of \eqref{20171012_2:16am} we arrive at
\begin{equation*}
\left(\frac{\Lambda_{2,6}(0)}{\beta(2)\, \Lambda_{2,\infty}(0)}\right)^6\,\frac{1}{2}\leq \frac{\Gamma(q+1)}{\beta(2)^q \, \Gamma(q+2)}.
\end{equation*}
From \eqref{20171009_10:14pm} we have $\Lambda_{2,\infty}(0) = 1$ and the inequality above simplifies to
\begin{equation*}
\frac{\Lambda^6_{2,6}(0)}{2} \leq \frac{1}{\beta(2)^{q-6} (q+1)}.
\end{equation*}
The approximation $\Lambda^6_{2,6}(0) =  0.3368280$ was computed in \cite[Table 1]{CFOST15}, accurate to $5 \times 10^{-7}$. With this in hand, a routine numerical evaluation yields
$$q_0(2) \leq 6.76.$$
This concludes the proof of Theorem \ref{Thm3}.

\section{The Tomas--Stein case}

In this section we prove Theorem \ref{Thm4}.

\subsection{Preliminaries}
We need one last fact from the theory of Bessel functions which consists of a uniform pointwise bound.

\begin{lemma}[Bessel facts III]\label{Bessel3}
For all $\nu\geq\frac12$ and $r>\frac32\nu$ we have
\begin{equation}\label{UniformBound}
|J_\nu(r)|\leq r^{-1/2}.
\end{equation}
\end{lemma}

\begin{proof}
This follows from \cite[Theorem 3]{Kr14}.
\end{proof}

\subsection{Numerics}
Let $(d,q)\in\{(4,\frac{10}3),(5,3)\}$.
In light of part (i) of Theorem \ref{Thm1}, it suffices to show 
\begin{equation}\label{Goal}
\Lambda_{d,q}(k)<\Lambda_{d,q}(0), \text{ for all } k\geq 1.
\end{equation}
We shall accomplish this in three steps:
\begin{itemize}
\item[(a)] Determine a good numerical lower bound for the quantity $\Lambda^q_{d,q}(0)$;
\item[(b)] Invoke the upper bound from Theorem \ref{Thm2} in order to discard all but finitely many cases $1\leq k\leq K$, for some $K=K(d,q)<\infty$; 
\item[(c)] Determine a good numerical upper bound for $\Lambda^q_{d,q}(k)$, when $1\leq k\leq K$.
\end{itemize}

\noindent{\sc Remark}: Step (b) above relies on the critical exponent for Theorem \ref{Thm2} being less than the Tomas--Stein exponent, $2d/(d-\frac43)<2(d+1)/(d-1)$.

\subsubsection{The case $(d,q)=(4,\frac{10}3)$}
In order to find a good lower bound for the integral
$$\Lambda_{4,\frac{10}3}^{\frac{10}3}(0)=\int_0^\infty \big|r^{-1}J_1(r)\big|^{\frac{10}3} \,r^{3}\,\d r,$$
we split the region of integration into two pieces, $0\leq r\leq 200$ and $r>200$.
A routine numerical evaluation (e.g. on  {\it Mathematica}) of the integral
\begin{equation}\label{NumericalIntegral}
\int_0^{200} \big|r^{-1}J_1(r)\big|^{\frac{10}3} \,r^{3}\,\d r
\end{equation}
 returns the value $0.257$ (3 d.p.).\footnote{This means that the exact value of the integral \eqref{NumericalIntegral} lies in the interval $[0.257,0.258).$}
 To estimate the tail integral, we use Lemma \ref{Bessel3}.
 This yields the {\it exact} upper bound
 $$\int_{200}^\infty \big|r^{-1}J_1(r)\big|^{\frac{10}3} r^{3}\d r
 \leq\int_{200}^\infty \big|r^{-1}r^{-\frac12}\big|^{\frac{10}3} r^{3}\d r=0.005.$$
 
 The estimate \eqref{20171009_6:32pm} then reduces the proof of the case $(d,q)=(4,\frac{10}3)$ of Theorem \ref{Thm4} to a numerical verification of all those $k$ for which 
 \begin{equation}\label{howmany1}
U_{4,\frac{10}3}^{\frac{10}3}(k)= \left(\frac{{\bf L}^{\frac32(5k+6)}}{2^{k+1} \Gamma(k+2)}\right)^{\frac4{9k+12}}\left(\frac94+\frac{1}{4+\frac{10k}3}\right)>0.257-0.005.
 \end{equation}
  One easily checks that inequality \eqref{howmany1} holds if and only if $k\leq 28$.
  
  In the range $1\leq k\leq 28$, the integrals
  $$\Lambda_{4,\frac{10}3}^{\frac{10}3}(k)=\int_0^\infty \big|r^{-1}J_{k+1}(r)\big|^{\frac{10}3} \,r^{3}\, \d r$$
  are treated numerically on the interval $0\leq r\leq 200$, and estimated crudely when $r>200$.
Numerical evaluation of the integrals
$$\int_0^{200} \big|r^{-1}J_{k+1}(r)\big|^{\frac{10}3} \,r^{3}\, \d r$$
   returns the following values (3 d.p.), for each $1\leq k\leq 28$:\\

\noindent {{0.146}, {0.103}, {0.080}, {0.066}, {0.056}, 
{0.048}, {0.043}, {0.038}, {0.035}, {0.032}, 
{0.029}, {0.027}, {0.025}, {0.024}, {0.022}, 
{0.021}, {0.020}, {0.019}, {0.018}, {0.017}, 
{0.016}, {0.016}, {0.015}, {0.014}, {0.014}, 
{0.013}, {0.013}, {0.012}.\\


\noindent The tails are again estimated with \eqref{UniformBound}.
For every $1\leq k\leq 28$, note that $\frac32(k+1)<\frac32(29+1)<200$, 
so that the assumptions of Lemma \ref{Bessel3} are verified,
and  therefore
 $$\int_{200}^\infty \big|r^{-1}J_{k+1}(r)\big|^{\frac{10}3} \,r^{3}\,\d r
 \leq\int_{200}^\infty \big|r^{-1}r^{-\frac12}\big|^{\frac{10}3} \,r^{3}\,\d r=0.005.$$
 In particular, \eqref{Goal} holds when $(d,q)=(4,\frac{10}3)$.

\subsubsection{The case $(d,q)=(5,3)$}
In order to find a good lower bound for the integral
$$\Lambda_{5,3}^3(0)=\int_0^\infty \big|r^{-\frac32}J_{\frac32}(r)\big|^3 \,r^{4}\,\d r,$$
we split the region of integration into the same two pieces.
Numerical evaluation of the integral
$$\int_0^{200} \big|r^{-\frac32}J_{\frac32}(r)\big|^3 \,r^{4}\,\d r$$
 returns the value $0.210$ (3 d.p.).
 To estimate the tail integral, we use Lemma \ref{Bessel3}.
 This yields the exact upper bound
 $$\int_{200}^\infty \big|r^{-\frac32}J_{\frac32}(r)\big|^3 \,r^4\,\d r
 \leq\int_{200}^\infty \big|r^{-\frac32}r^{-\frac12}\big|^3\, r^4\,\d r=0.005.$$
 
 The estimate \eqref{20171009_6:32pm} then reduces the proof of the case $(d,q)=(5,3)$ of Theorem \ref{Thm4} to a numerical verification of all those $k$ for which 
 \begin{equation}\label{howmany}
 U_{5,3}^3(k)=\left(\frac{{\bf L}^{6k+10}}{2^{k+\frac32} \Gamma(k+\frac52)}\right)^{\frac1{2k+\frac{11}3}}\left(2+\frac{1}{5+3k}\right)>0.210-0.005.
 \end{equation}
  One easily checks that inequality \eqref{howmany} holds if and only if $k\leq 28$.
  
  In the range $1\leq k\leq 28$, the integrals
  $$\Lambda_{5,3}^3(k)=\int_0^\infty \big|r^{-\frac32}J_{k+\frac32}(r)\big|^3 \,r^{4}\,\d r$$
  are treated numerically on the interval $0\leq r\leq 200$, and estimated crudely when $r>200$.
  Numerical evaluation of the integrals
  $$\int_0^{200} \big|r^{-\frac32}J_{k+\frac32}(r)\big|^3 \,r^{4}\,\d r$$
  returns the following values (3 d.p.), for each $1\leq k\leq 28$:\\

\noindent {{0.134}, {0.099}, {0.079}, {0.066}, {0.056}, 
{0.049}, {0.044}, {0.039}, {0.036}, {0.033},
{0.030}, {0.028}, {0.026}, {0.024},
{0.023}, {0.022}, {0.020}, {0.019}, 
{0.018}, {0.017}, {0.017}, {0.016}, {0.015}, {0.015}, {0.014}, {0.014}, {0.013}, {0.013}.\\


\noindent The tails are again estimated with Lemma \ref{Bessel3}.
For every $1\leq k\leq 28$, recall that $\frac32(k+1)<200$, 
and finally conclude
 $$\int_{200}^\infty \big|r^{-\frac32}J_{k+1}(r)\big|^3 \,r^{4}\,\d r
 \leq\int_{200}^\infty \big|r^{-\frac32}r^{-\frac12}\big|^{3}\, r^4\,\d r=0.005.$$
 In particular, \eqref{Goal} holds when $(d,q)=(5,3)$.
 This concludes the proof of Theorem \ref{Thm4}.

\section*{Acknowledgments}
The software {\it Mathematica} was used to perform the numerical tasks described in this paper. The authors thank Dimitar Dimitrov for a careful reading of the manuscript and for having suggested the use of the quadratic approximation \eqref{Quadratic_approx} in the argument of Section \ref{Sec4_Infinity}. The authors are also thankful to \'{A}rp\'{a}d Baricz, Doron Lubinsky, Jos\'{e} Madrid, Fernando Rodriguez-Villegas and Krzysztof Stempak for helpful comments and discussions. E.C. acknowledges support from CNPq - Brazil, FAPERJ - Brazil and the Simons Associate Scheme from the International Centre for Theoretical Physics (ICTP) - Italy. D.O.S. was partially supported by the Hausdorff Center for Mathematics and DFG grant CRC 1060. M.S. acknowledges support from FAPERJ - Brazil. Part of this work was carried out during a research visit to ICTP - Italy. The authors thank the warm hospitality of the Institute.

\end{document}